\documentclass{amsart}

\usepackage[margin=1.5in]{geometry}
\usepackage{graphicx}
\usepackage{morefloats,bm,caption,amsbsy,enumerate,amsmath,amsthm,
amssymb,mathtools,amsfonts,multirow,verbatim,tikz,tikz-cd,thmtools,thm-restate}
\usepackage{enumitem}
\usepackage[alphabetic]{amsrefs}
\usepackage{chngcntr}

\usepackage[hidelinks,colorlinks=true,linkcolor=blue, citecolor=black,linktocpage=true]{hyperref}
\usetikzlibrary{matrix,arrows,decorations.pathmorphing}
\usepackage{url}

\usepackage{xcolor}

\counterwithin{figure}{section}

\newtheorem{thm}{Theorem}[section]

\newtheorem{cor}[thm]{Corollary}
\newtheorem{lem}[thm]{Lemma}

\theoremstyle{definition}
\newtheorem{define}[thm]{Definition}

\theoremstyle{remark}
\newtheorem{rem}[thm]{Remark}

\newcommand{\ve}[1]{\boldsymbol{\mathbf{#1}}}
\newcommand{\R}{\mathbb{R}}

\newcommand{\Z}{\mathbb{Z}}

\renewcommand{\d}{\partial}
\renewcommand{\subset}{\subseteq}

\renewcommand{\bar}{\overline}

\newcommand{\iso}{\cong}

\DeclareMathOperator{\gr}{{gr}}

\DeclareMathOperator{\id}{{id}}

\DeclareMathOperator{\Spin}{{Spin}}

\DeclareMathOperator{\Sym}{{Sym}}

\newcommand{\bF}{\mathbb{F}}

\newcommand{\bK}{\mathbb{K}}
\newcommand{\bL}{\mathbb{L}}

\newcommand{\bO}{\mathbb{O}}

\newcommand{\bT}{\mathbb{T}}

\newcommand{\cA}{\mathcal{A}}

\newcommand{\cC}{\mathcal{C}}
\newcommand{\cD}{\mathcal{D}}

\newcommand{\cF}{\mathcal{F}}

\newcommand{\cM}{\mathcal{M}}

\newcommand{\frs}{\mathfrak{s}}
\newcommand{\frt}{\mathfrak{t}}

\newcommand{\cCFK}{\mathcal{C\!F\!K}}

\newcommand{\CFK}{\mathit{CFK}}
\newcommand{\HFK}{\mathit{HFK}}

\newcommand{\HFKh}{\widehat{\HFK}}

\newcommand{\xs}{\ve{x}}
\newcommand{\ys}{\ve{y}}
\newcommand{\zs}{\ve{z}}
\newcommand{\ws}{\ve{w}}

\newcommand{\as}{\ve{\alpha}}
\newcommand{\bs}{\ve{\beta}}

\renewcommand{\a}{\alpha}
\renewcommand{\b}{\beta}

\usepackage{leftidx}

\title[Knot Floer and strongly homotopy-ribbon concordances]{Knot Floer homology and strongly homotopy-ribbon concordances}
\author{Maggie Miller}
\address{Department of Mathematics\\Princeton University\\  Princeton, NJ 08544, USA}
\email{maggiem@math.princeton.edu}
\author{Ian Zemke}
\address{Department of Mathematics\\Princeton University\\  Princeton, NJ 08544, USA}

\email{izemke@math.princeton.edu}
\thanks{MM was supported by NSF grant DGE-1656466. IZ was supported by NSF grant DMS-1703685. }

\begin{document}
\maketitle
	\begin{abstract}
		We prove that the map on knot Floer homology induced by a strongly homotopy-ribbon concordance is injective. One application is that the Seifert genus is monotonic under strongly homotopy-ribbon concordance.
		\end{abstract}
\section{Introduction}

A \emph{concordance} $C$ from $K_0$ to $K_1$ is a smoothly embedded annulus in $[0,1]\times S^3$ such that $\d C=-K_0\cup K_1$. A \emph{ribbon concordance} from $K_0$ to $K_1$ is a concordance $C$ such that the projection of $C$ to $[0,1]$ is Morse and has only index 0 and 1 critical points.  

If $C$ is a concordance from $K_0$ to $K_1$, write $\pi_1(C)$, $\pi_1(K_0)$ and $\pi_1(K_1)$ for the fundamental groups of the complements. A \emph{homotopy-ribbon concordance} is a concordance such that 
\[
\pi_1(K_0)\to \pi_1(C)\qquad \text{is an injection,}\qquad \text{and} \qquad \pi_1(K_1)\to \pi_1(C)\qquad \text{is a surjection.}
\]
The definition is justified by a result of Gordon~\cite{Gordon}*{Lemma~3.1} which implies that ribbon concordances are homotopy-ribbon.

In fact, Gordon showed something stronger: if $C$ is a ribbon concordance, then its complement can be built by attaching only 1-handles and 2-handles to $S^3\setminus K_0$.


Justified by Gordon's observation, one can make the following definition:

\begin{define}
A \emph{strongly homotopy-ribbon concordance} $C$ from $K_0$ to $K_1$ is a concordance such that the complement of $C$ can be constructed by attaching 4-dimensional 1-handles and 2-handles to $S^3\setminus K_0$.
\end{define}

Strongly homotopy-ribbon concordance has been studied by, e.g., Cochran~\cite{CochranRibbon} and Larson-Meier~\cite{LarsonFiberedRibbon}.

By definition, we have
\[
\{\text{ribbon concordances}\} \subset \{\text{strongly homotopy-ribbon concordances}\}
\]
\[
\subset \{\text{homotopy-ribbon concordances}\}\subsetneq \{\text{concordances}\}.
\]

It is unknown whether the first two inclusions are strict.

In this paper, we show that knot Floer homology obstructs strongly homotopy-ribbon concordance. Our argument is based on a result of the second author~\cite{ZemRibbon} for ribbon concordance.

\subsection{Knot Floer homology and strongly homotopy-ribbon concordances}

Knot Floer homology is an invariant of knots in 3-manifolds discovered independently by Ozsv\'{a}th and Szab\'{o}~\cite{OSKnots} and Rasmussen~\cite{RasmussenKnots}. 


If $K$ is a knot in $S^3$, the simplest version of knot Floer homology is the hat version, which is a bigraded $\bF_2$ vector space:
\[
\HFKh(K)=\bigoplus_{i,j\in \Z} \HFKh_i(K,j).
\]
The index $i$ denotes the Maslov grading, and $j$ the Alexander grading.

Associated to a concordance $C$ from $K_0$ to $K_1$, Juh\'{a}sz and Marengon~\cite{JMConcordance} described a bigraded homomorphism
\[
\widehat{F}_C\colon \HFKh(K_0)\to \HFKh(K_1),
\]
which is well defined up to bigraded automorphisms of $\HFKh(K_0)$ and $\HFKh(K_1)$. The ambiguity can be eliminated by picking a decoration of the concordance $C$ consisting of a pair of disjoint arcs on $C$ running from $K_0$ to $K_1$. Their construction used a more general construction of link cobordism maps described by Juh\'{a}sz ~\cite{JCob} for link Floer homology.

There is a more general version of knot Floer homology, the \empty{infinity version}, denoted
\[
\CFK^\infty(K),
\]
which is a $\Z\oplus \Z$-filtered, bigraded chain complex over the ring $\bF_2[U,U^{-1}]$. We refer the reader to \cite{Man-Intro-HFK} for a nice introduction to the numerous versions of knot Floer homology.

The second author~\cite{ZemCFLTQFT} constructed a bigraded, filtered, $\bF_2[U,U^{-1}]$-equivariant cobordism map
\[
F_{C}\colon \CFK^\infty(K_0)\to \CFK^\infty(K_1).
\]

The second author showed that the cobordism map induced by a ribbon concordance is left invertible~\cite{ZemRibbon}*{Theorem~1.1}.  Our main result is that the same is true for strongly homotopy-ribbon concordances:

\begin{thm}\label{thm:1} If $C$ is a strongly homotopy-ribbon concordance from $K_0$ to $K_1$, then the induced map
\[
F_{C}\colon \CFK^\infty(K_0)\to \CFK^\infty(K_1)
\]
admits a left inverse. In particular, the induced map
\[
\widehat{F}_{C}\colon \HFKh(K_0)\to \HFKh(K_1)
\]
is an injection.
\end{thm}

Our proof uses a similar doubling trick as~\cite{ZemRibbon}: if $\bar{C}$ denotes the concordance from $K_1$ to $K_0$ obtained by turning around and reversing the orientation of $C$, then we will show that
\[
F_{\bar{C}}\circ F_{C}\simeq \id_{\CFK^\infty(K_0)}.
\]

\begin{rem}\label{remark:h}
In fact, the proof of Theorem~\ref{thm:1} works for strongly homotopy-ribbon concordances in an arbitrary $h$-cobordism of $S^3$.  See Remark \ref{rem:end}.  
This is potentially interesting because one may allow a homotopy-ribbon disk to live in any homotopy $B^4$. See, e.g., ~\cite{CassonLoop}.
\end{rem}

Ozsv\'{a}th and Szab\'{o}~\cite{OSgenusbounds} proved that knot Floer homology detects the Seifert genus:
\[
g_3(K)=\max\left\{j\colon \HFKh(K,j)\neq \{0\}\right\}.
\]

Combined with our Theorem~\ref{thm:1}, we obtain the following generalization of~\cite{ZemRibbon}*{Theorem~1.5}:

\begin{cor}\label{cor:1} If there is a strongly homotopy-ribbon concordance from $K_0$ to $K_1$, then
\[
g_3(K_0)\le g_3(K_1).
\]
\end{cor}

By Remark \ref{remark:h}, Corollary \ref{cor:1} applies whenever there is a strongly homotopy-ribbon concordance from $K_0$ to $K_1$ in any $h$-cobordism of $S^3$.

\section{Strongly homotopy-ribbon concordances}
\label{sec:ribbon=>stronghomotopy}

In this section, we demonstrate that ribbon concordances are strongly homotopy-ribbon concordances. This is described in the proof of~\cite{Gordon}*{Lemma~3.1}, however we give a proof for the reader's convenience.

\begin{lem}
If $C$ is a ribbon concordance from $K_0$ to $K_1$, then $C$ is a strongly homotopy-ribbon concordance.
\end{lem}
\begin{proof} Suppose $C$ is a ribbon concordance.  For notational simplicity, we restrict to the case when the concordance has a single 0-handle and 1-handle. The concordance $C$ has a movie consisting of a single birth followed by a saddle, which is shown in Figure~\ref{fig::6}. The birth adds an unknot $U$, and the saddle adds a band $B$.

\begin{figure}[ht!]
	\centering
	\scalebox{.8}{
\begingroup%
  \makeatletter%
  \providecommand\color[2][]{%
    \errmessage{(Inkscape) Color is used for the text in Inkscape, but the package 'color.sty' is not loaded}%
    \renewcommand\color[2][]{}%
  }%
  \providecommand\transparent[1]{%
    \errmessage{(Inkscape) Transparency is used (non-zero) for the text in Inkscape, but the package 'transparent.sty' is not loaded}%
    \renewcommand\transparent[1]{}%
  }%
  \providecommand\rotatebox[2]{#2}%
  \newcommand*\fsize{\dimexpr\f@size pt\relax}%
  \newcommand*\lineheight[1]{\fontsize{\fsize}{#1\fsize}\selectfont}%
  \ifx\svgwidth\undefined%
    \setlength{\unitlength}{431.1659958bp}%
    \ifx\svgscale\undefined%
      \relax%
    \else%
      \setlength{\unitlength}{\unitlength * \real{\svgscale}}%
    \fi%
  \else%
    \setlength{\unitlength}{\svgwidth}%
  \fi%
  \global\let\svgwidth\undefined%
  \global\let\svgscale\undefined%
  \makeatother%
  \begin{picture}(1,0.25534589)%
    \lineheight{1}%
    \setlength\tabcolsep{0pt}%
    \put(0,0){\includegraphics[width=\unitlength,page=1]{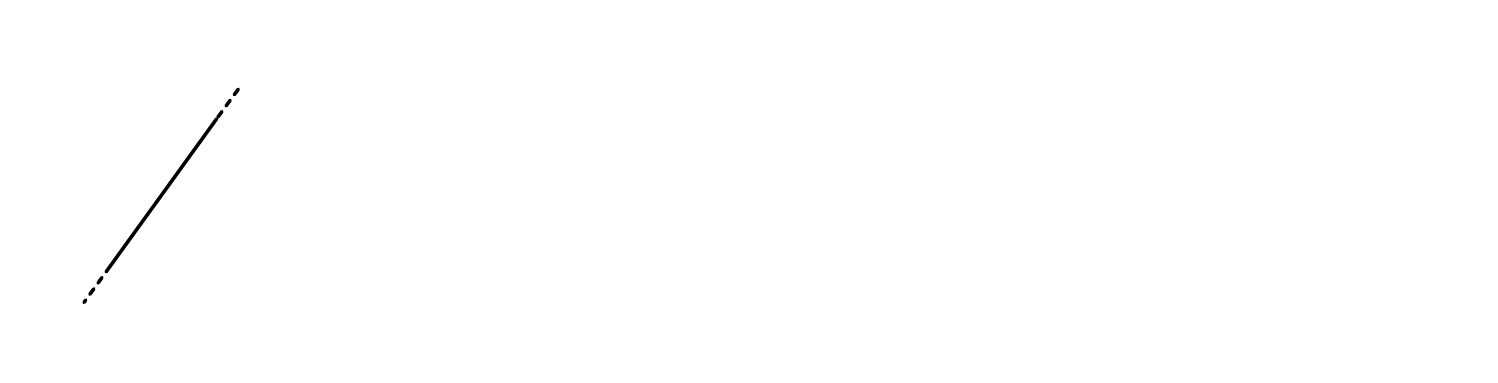}}%
    \put(0.09166478,0.07293964){\color[rgb]{0,0,0}\makebox(0,0)[lt]{\lineheight{1.25}\smash{\begin{tabular}[t]{l}$K_0$\end{tabular}}}}%
    \put(0,0){\includegraphics[width=\unitlength,page=2]{fig6.pdf}}%
    \put(0.47387999,0.01785239){\color[rgb]{0,0,0}\makebox(0,0)[t]{\lineheight{1.25}\smash{\begin{tabular}[t]{c}birth\end{tabular}}}}%
    \put(0,0){\includegraphics[width=\unitlength,page=3]{fig6.pdf}}%
    \put(0.83806305,0.01785239){\color[rgb]{0,0,0}\makebox(0,0)[t]{\lineheight{1.25}\smash{\begin{tabular}[t]{c}band\end{tabular}}}}%
    \put(0,0){\includegraphics[width=\unitlength,page=4]{fig6.pdf}}%
    \put(0.54012459,0.14269644){\color[rgb]{0,0,0}\makebox(0,0)[lt]{\lineheight{1.25}\smash{\begin{tabular}[t]{l}$U$\end{tabular}}}}%
    \put(0.87712466,0.17881656){\color[rgb]{0,0,0}\makebox(0,0)[lt]{\lineheight{1.25}\smash{\begin{tabular}[t]{l}$B$\end{tabular}}}}%
    \put(0,0){\includegraphics[width=\unitlength,page=5]{fig6.pdf}}%
  \end{picture}%
\endgroup%
}
	\caption{A movie for a ribbon concordance with a single birth and band.}\label{fig::6}
\end{figure}

In Figure~\ref{fig::7}, another movie of a concordance is shown, featuring a  4-dimensional 1-handle and 2-handle, but no births or saddles. Write $S^0$ for the attaching sphere of the 1-handle, and $k$ for the attaching sphere of the 2-handle. In Figure~\ref{fig::7}, we use the dotted unknot notation for 1-handle attachment. We note that the transition through frames 3, 4 and 5 in Figure~\ref{fig::6} is achieved by a sequence of isotopies (we remind the reader that handlesliding the knot $K_0$ over the knot $k$ or the dotted unknot corresponding to $S^0$ corresponds to an isotopy of the knot $K_0$ in the 3-manifold obtained by surgering $S^3$ on $S^0$ and $k$).

\begin{figure}[ht!]
	\centering
	\scalebox{.8}{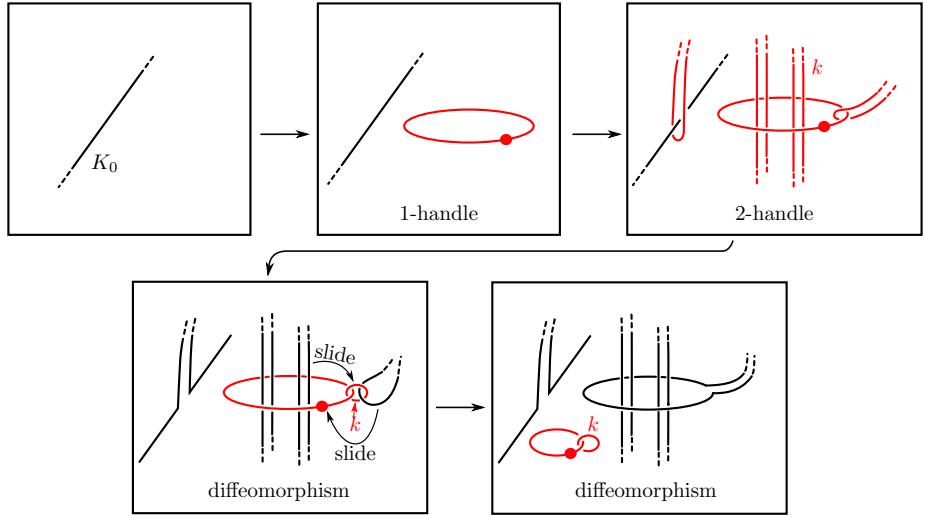}
	\caption{A movie for a ribbon concordance with a single 1-handle and 2-handle.}\label{fig::7}
\end{figure}

We claim that the movies shown in Figure~\ref{fig::6} and Figure~\ref{fig::7} can be related by a set of 4-dimensional movie moves. This is achieved by taking the movie in Figure~\ref{fig::6} and adjoining a canceling pair of 4-dimensional 1- and 2-handles. Write $k$ for the 2-handle, which we assume is a 0-framed unknot. After sliding $U$ over the dotted unknot forming the 1-handle, and sliding the band $B$ over $k$ (repeatedly), the band $B$ now cancels the birth which added $U$, and we are left with the movie in Figure~\ref{fig::7}.
\end{proof}

\section{A handle decomposition of the doubled concordance}

We now describe a handle decomposition of the complement of the doubled concordance $\bar{C}\circ C$, when $C$ is strongly homotopy-ribbon.

\begin{lem}\label{lem:handle-decomp} Suppose that $C$ is a strongly homotopy-ribbon concordance from $K_0$ to $K_1$, and $C$ has a handle decomposition consisting of $n$ 1-handles, attached along 0-spheres $S^0_1,\dots, S^0_n$, and $n$ 2-handles, attached along framed knots $k_1,\dots, k_n$. A handle decomposition for the complement of $\bar{C}\circ C\subset [0,1]\times S^3$ is as follows:
\begin{enumerate}
\item $n$ 1-handles, attached along  $S^0_1,\dots, S_n^0$.
\item $n$ 2-handles, attached along $k_1,\dots, k_n$.
\item $n$ 2-handles, attached along 0-framed meridians $k_1',\dots, k_n'$ of $k_1,\dots, k_n$.
\item $n$ 3-handles, attached along 2-spheres obtained by taking the belt spheres of the 1-handles, removing disks where they intersect $k_1,\dots, k_n$, and adding a collection of pairwise disjoint annuli which connect the punctures to the 0-framed meridians $k_1',\dots, k_n'$, and are contained in a neighborhood of the knots $k_1,\dots, k_n$. See Figure~\ref{fig::4} for an example.
\end{enumerate}
\end{lem}

\begin{figure}[ht!]
	\centering
\begingroup%
  \makeatletter%
  \providecommand\color[2][]{%
    \errmessage{(Inkscape) Color is used for the text in Inkscape, but the package 'color.sty' is not loaded}%
    \renewcommand\color[2][]{}%
  }%
  \providecommand\transparent[1]{%
    \errmessage{(Inkscape) Transparency is used (non-zero) for the text in Inkscape, but the package 'transparent.sty' is not loaded}%
    \renewcommand\transparent[1]{}%
  }%
  \providecommand\rotatebox[2]{#2}%
  \newcommand*\fsize{\dimexpr\f@size pt\relax}%
  \newcommand*\lineheight[1]{\fontsize{\fsize}{#1\fsize}\selectfont}%
  \ifx\svgwidth\undefined%
    \setlength{\unitlength}{256.45900902bp}%
    \ifx\svgscale\undefined%
      \relax%
    \else%
      \setlength{\unitlength}{\unitlength * \real{\svgscale}}%
    \fi%
  \else%
    \setlength{\unitlength}{\svgwidth}%
  \fi%
  \global\let\svgwidth\undefined%
  \global\let\svgscale\undefined%
  \makeatother%
  \begin{picture}(1,0.32187908)%
    \lineheight{1}%
    \setlength\tabcolsep{0pt}%
    \put(0,0){\includegraphics[width=\unitlength,page=1]{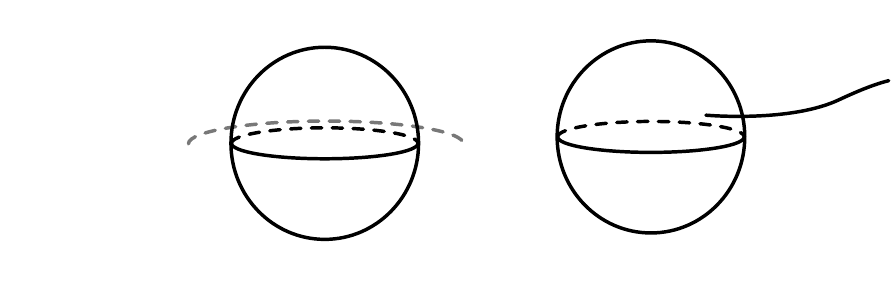}}%
    \put(0.12510824,0.21996487){\color[rgb]{0,0,0}\makebox(0,0)[rt]{\lineheight{1.25}\smash{\begin{tabular}[t]{r}$k_i'$\end{tabular}}}}%
    \put(0.91596937,0.21651512){\color[rgb]{0,0,0}\makebox(0,0)[rt]{\lineheight{1.25}\smash{\begin{tabular}[t]{r}$k_i$\end{tabular}}}}%
    \put(0,0){\includegraphics[width=\unitlength,page=2]{fig5.pdf}}%
    \put(0.48440393,0.28559245){\color[rgb]{0.31372549,0.31372549,0.31372549}\makebox(0,0)[lt]{\lineheight{1.25}\smash{\begin{tabular}[t]{l}$S_j^2$\end{tabular}}}}%
    \put(0,0){\includegraphics[width=\unitlength,page=3]{fig5.pdf}}%
    \put(0.04829793,0.19518339){\color[rgb]{0,0,0}\makebox(0,0)[rt]{\lineheight{1.25}\smash{\begin{tabular}[t]{r}$k_i$\end{tabular}}}}%
  \end{picture}%
\endgroup%

	\caption{Examples of the 2-spheres $S^2_j$ (shaded gray), along which the 3-handles are attached. In general, a single knot $k_i$ may intersect $S_j'$ many times, and also, many different $k_i$ may intersect a single 2-sphere $S_j^2$.}\label{fig::4}
\end{figure}

\begin{proof} The proof is  standard, though it is usually presented in the context of closed $4$-manifolds, where the $3$-handles need not be described explicitly. See~\cite{GompfStipsicz}*{Example~4.6.3}. For notational simplicity, we focus on the case when the complement of $C$ has a handle decomposition with one 1-handle and one 2-handle. The general case is a straightforward modification.

Let $S^0$ denote the attaching 0-sphere of the 1-handle. Write $S^3(S^0)$ for the 3-manifold obtained from $S^3$ by surgery along $S^0$.  (In general, we write $Y(j)$ for the manifold obtained by surgery on a framed sphere $j$ in a 3-manifold $Y$). Note that $S^3(S^0)\cong S^1\times S^2$. Let $k\subset S^3(S^0)$ be the knot along which the 2-handle is attached. The framing of $k$ induces an identification of a neighborhood of $k$ with $S^1\times D^2$. Write $N$ for this neighborhood. The surgered manifold $S^3(S^0)(k)$ is obtained from $S^3(S^0)$ by removing $N=S^1\times D^2$ and gluing in $N'=D^2\times S^1$.  The belt sphere of the 2-handle is the knot 
\[
k'=\{0\}\times S^1,
\]
which is framed using the same identification of $N'$ as $D^2\times S^1$.

The surgered manifold $S^3(S^0)(k)(k')$ is canonically diffeomorphic to $S^3(S^0)$. The attaching sphere of the 3-handle in $S^3(S^0)(k)(k')$ is easy to see with this description: it coincides with the belt sphere of 1-handle, under the identification of $S^3(S^0)(k)(k')$ with $S^3(S^0)$.

We consider now the image of the attaching 2-sphere of the 3-handle in $S^3(S^0)(k)$. Outside of $N$ (the neighborhood of $k$), the attaching sphere of the 3-handle coincides with the belt sphere of the 1-handle attached along $S^0$. Inside of $N'$, it is equal to a collection of annuli of the form $a\times S^1$, where $a\subset D^2$ is a radial arc extending from $\{0\}$ to $\d D^2$. There is one annulus for each intersection point of $k$ with the belt sphere of the 1-handle.

Next, we push $k'$ out of the neighborhood $N'$, so that we obtain a more standard Kirby calculus picture. When we do this, it is straightforward to check that $k'$ becomes a 0-framed meridian, and that the annuli are as described.
\end{proof}

\section{Proof of the main theorem}

Our proof follows from a ``sphere tubing'' property of the link Floer TQFT from \cite{ZemCFLTQFT}, which we state below in Lemma~\ref{lem:sphere-slide}. We first review some background about the link Floer TQFT.

The link Floer TQFT uses the following decorated link cobordism category, originally described by Juh\'{a}sz \cite{JCob}:

\begin{define}
\begin{enumerate}
\item A \emph{multi-based link} $\bL=(L,\ws,\zs)$ in a 3-manifold $Y$ is an orientated link $L$, containing two finite collections of basepoints, $\ws$ and $\zs$, such that each component of $L$ contains at least one $\ws$ basepoint and one $\zs$ basepoint. Furthermore, as one traverses $L$, the basepoints alternate between $\ws$ and $\zs$.
\item A \emph{decorated link cobordism} $(W,\cF)$ from $(Y_1,\bL_1)$ to $(Y_2,\bL_2)$ consists of a cobordism $W$ from $Y_1$ to $Y_2$, as well as a decorated surface $\cF=(\Sigma,\cA)$, as follows. The surface $\Sigma$ is a properly and smoothly embedded surface in $W$, such that $\d \Sigma=-L_1\cup L_2$. Furthermore, $\cA\subset\Sigma$ is a properly embedded 1-manifold  which divides $\Sigma$ into two subsurfaces, $\Sigma_{\ws}$ and $\Sigma_{\zs}$, which satisfy $\ws_1,\ws_2\subset \Sigma_{\ws}$ and $\zs_1,\zs_2\subset \Sigma_{\zs}$. 
\end{enumerate}
\end{define}

The link Floer TQFT from \cite{ZemCFLTQFT} assigns cobordism maps to a slightly different version of knot Floer homology than we described in the introduction. If $(Y,\bK)$ is a multi-pointed knot and $\frs\in \Spin^c(Y)$, we consider a version of knot Floer homology which we denote by $\cCFK^-(Y,\bK,\frs)$. Explicitly, if $(\Sigma,\as,\bs,\ws,\zs)$ is a Heegaard diagram, then $\cCFK^-(Y,\bK,\frs)$ is freely generated over the 2-variable polynomial ring $\bF_2[u,v]$ by intersection points $\xs\in \bT_{\a}\cap \bT_{\b}$ with $\frs_{\ws}(\xs)=\frs$. The differential counts Maslov index 1 holomorphic disks in the symmetric product $\Sym^n(\Sigma)$ (where $n=g(\Sigma)+|\ws|-1$) via the formula
\[
\d \xs=\sum_{\ys\in \bT_{\a}\cap \bT_{\b}} \sum_{\substack{\phi\in \pi_2(\xs,\ys)\\ \mu(\phi)=1}} \# \left(\cM(\phi)/\R\right ) \cdot u^{n_{\ws}(\phi)} v^{n_{\zs}(\phi)}\cdot \ys.
\]
 In the above expression, $n_{\ws}(\phi)$ denotes the total multiplicity of a class $\phi$ over the $\ws$ basepoints, and $n_{\zs}(\phi)$ is defined similarly.

 The chain complex $\cCFK^-(Y,\bK,\frs)$ has a $\Z\oplus \Z$ filtration given by powers of $u$ and $v$. If $c_1(\frs)$ is torsion and $K$ is null-homologous, there are also three gradings: two Maslov gradings, $\gr_{\ws}$ and $\gr_{\zs}$, as well as an Alexander grading $A$. The gradings are related by the equation $A=\tfrac{1}{2}(\gr_{\ws}-\gr_{\zs})$.  The actions of $u$ and $v$ have $(\gr_{\ws},\gr_{\zs})$-bigradings of $(-2,0)$ and $(0,-2)$, respectively.

For a knot $K$ in $S^3$, the previously described knot Floer complex $\CFK^\infty(K)$ is obtained by decorating $K$ with two basepoints to form $\bK$, inverting $u$ and $v$ in $\cCFK^-(S^3,\bK)$, and taking the subcomplex  in Alexander grading zero. The action of $U$ on $\CFK^\infty$ corresponds to the action of the product $uv$ on $\cCFK^-$. The differential drops both Maslov gradings by 1, and preserves the Alexander grading.

If $(W,\cF)$ is a decorated link cobordism from $(Y_1,\bK_1)$ to $(Y_2,\bK_2)$, and $\frs\in \Spin^c(W)$, then the construction from \cite{ZemCFLTQFT} gives a map
\[
F_{W,\cF,\frs}\colon \cCFK^-(Y_1,\bK_1,\frs|_{Y_1})\to \cCFK^-(Y_2,\bK_2,\frs|_{Y_2}).
\]
When $W$ and $\frs$ are understood from context, we will simply write $F_{\cF}$ for the cobordism map.
 
 If $(W,\cF)$ is a concordance in $[0,1]\times S^3$, decorated with a pair of arcs, the link cobordism maps preserve the Maslov and Alexander gradings by \cite{ZemAbsoluteGradings}*{Theorem~1.4}. Correspondingly, the cobordism maps restrict to give maps on $\CFK^\infty$ (the subcomplex in Alexander grading 0), as described in the introduction.

Before we prove Theorem~\ref{thm:1}, we recall the key lemma proven by the second author~\cite{ZemRibbon}*{Lemma~3.1}:

\begin{lem}
\label{lem:sphere-slide} Suppose that $\cC=(C,\cA)$ is a decorated concordance in $[0,1]\times S^3$, and $S\subset [0,1]\times S^3$ is a smoothly embedded 2-sphere in the complement of $C$. Let $\cC'=(C',\cA')$ be a decorated concordance obtained by connecting $C$ and $S$ with a tube, away from the decorations. Then
\[
F_{\cC}\simeq F_{\cC'},
\]
as maps from $\CFK^\infty(K_0)$ to $\CFK^\infty(K_1)$.
\end{lem}
\begin{proof} Since $\CFK^\infty$ is obtained by inverting $u$ and $v$ in $\cCFK^-$, and then taking the subcomplex in Alexander grading zero, it suffices to prove the analogous claim on $\cCFK^-$.

 We factor the cobordism map $F_{\cC}$ through a closed neighborhood $N(S)$ of the sphere $S$. Here we use the fact that the tube meets $\partial N(S)\cong S^1\times S^2$ in an unknot. Write $\cD$ and $\cD'$ for intersections of $(C,\cA)$ and $(C',\cA')$ with $N(S)$. After an isotopy, we may assume that both $\cD$ and $\cD'$ are disks which intersect the dividing set in a single arc. We obtain cobordism maps 
\[
F_{N(S),\cD,\frt_0},\,F_{N(S),\cD',\frt_0}:\cCFK^-(\emptyset)\to\cCFK^-(S^1\times S^2,\bO,\mathfrak{s}_0),\]
 where $\bO$ is a doubly pointed unknot, $\frt_0$ is the $\Spin^c$ structure on $N(S)$ which evaluates trivially on $S$, and $\mathfrak{s}_0$ is its restriction to $\d N(S)$. (Note that by definition, $\cCFK^-(\emptyset):=\bF_2[u,v]$, with vanishing differential).  Since $\cC$ and $\cC'$ coincide outside of $N(S)$, using the composition law, it suffices to show that $F_{N(S),\cD,\frt_0}\simeq F_{N(S),\cD',\frt_0}$.

There is a chain isomorphism
\begin{equation}
\cCFK^-(S^1\times S^2,\bO,\frs_0)\iso \left((\bF_2)_{(-\frac{1}{2}, -\frac{1}{2})}\oplus (\bF_2)_{(\frac{1}{2}, \frac{1}{2})}\right)\otimes_{\bF_2} \bF_2[u,v],\label{eq:s1xs2-unknot}
\end{equation}
where we view the latter complex as having vanishing differential, and $(\bF_2)_{(p,q)}$ denotes the rank 1 vector space over $\bF_2$, concentrated in $(\gr_{\ws},\gr_{\zs})$-bigrading $(p,q)$.

The grading change formula from \cite{ZemAbsoluteGradings}*{Theorem~1.4} implies that $F_{N(S),\cD,\frt_0}$ and $F_{N(S),\cD',\frt_0}$ both have $(\gr_{\ws},\gr_{\zs})$-bigrading $\left(-\tfrac{1}{2},-\tfrac{1}{2}\right)$. Noting that $u$ and $v$ have $(\gr_{\ws}, \gr_{\zs})$ bigrading $(-2,0)$ and $(0,-2)$, from Equation~\eqref{eq:s1xs2-unknot} we find that $\cCFK^-(S^1\times S^2,\bO,\frs_0)$ has rank 1 in $(\gr_{\ws},\gr_{\zs})$-bigrading $\left(-\tfrac{1}{2},-\tfrac{1}{2}\right)$. Hence, if $F_{N(S),\cD,\frt_0}$ and $F_{N(S),\cD',\frt_0}$ are both non-zero, then they must be equal. By capping off with a 3-handle and a 4-handle, we obtain a 2-knot in $S^4$ decorated with a single dividing curve. The corresponding cobordism map is non-zero by \cite{ZemRibbon}*{Lemma~4.1}, so we conclude $F_{N(S),\cD,\frt_0}$ and $F_{N(S),\cD',\frt_0}$ are both non-zero, and hence equal.
\end{proof}

\begin{proof}[Proof of Theorem~\ref{thm:1}] 

See Figure~\ref{fig:example} for a schematic example of this proof.

Let $C$ be a strongly homotopy-ribbon concordance from $K_0$ to $K_1$, and let $\bar{C}$ be the concordance from $K_1$ to $K_0$ obtained by turning around and reversing the orientation of $C$. Lemma~\ref{lem:handle-decomp} gives a handle decomposition of the complement of $\bar{C}\circ C$. Let $\cC$ be $C$, decorated with a pair of dividing arcs, and let $\bar{\cC}$ denote $\bar{C}$, decorated with the mirrored decoration.

By the composition law for link cobordisms,
\[
F_{\bar{\cC}}\circ F_{\cC}\simeq F_{\bar{\cC}\circ \cC},
\]
so it suffices to show that $F_{\bar{\cC}\circ \cC}\simeq \id_{\CFK^\infty(K_0)}.$

We use the handle decomposition described in Lemma~\ref{lem:handle-decomp}. Note that if we ignore the knot $K_0$, then the handle decomposition describes $[0,1]\times S^3$. The knot $K_0$ may of course be tangled with the 2-handle attaching spheres $k_1,\dots, k_n$.

The idea is as follows. We  ``pull'' the knot $K_0$ away from the attaching spheres for the 4-dimensional handles. Of course to achieve this, we must pull $K_0$ through the framed knots $k_1,\dots, k_n$ (i.e. we must perform a sequence of crossing changes of $K_0$ with the knots $k_1,\dots, k_n$). To achieve a crossing change of $K_0$ with a framed knot $k_i$, we can isotope a segment of $K_0$ parallel to $k_i$ until we reach the 0-framed meridian $k_i'$, and then ``handleslide'' $K_0$ across $k_i$. Given the explicit description of the attaching spheres of the 3-handles stated in Lemma~\ref{lem:handle-decomp}, this isotopy and handleslide can be performed in the complement of the attaching 2-spheres of the 3-handles.

Note that handlesliding $K_0$ across $k_i'$ is not actually a handleslide of the described $4$-manifold, as $K_0$ is not a handle attaching circle. Rather, this move changes the concordance $\bar{C}\circ C$. Indeed the two concordances are related by tubing to $\bar{C}\circ C$ the 2-sphere $S_i$, which is isotopic to the union of a small Seifert disk of unknot $k_i'$ in $S^3$ together with a core of the 2-handle attached along $k_i'$. (That is, $S_i$ consists of the cocore of the $2$-handle attached along $k_i$ and the core of the $2$-handle attaching along $k_i'$, glued along their common boundary.) This move does not change the cobordism map by Lemma~\ref{lem:sphere-slide}.

After untangling $K_0$ from all of the attaching spheres of the 4-dimensional handles, we are left with the identity knot cobordism $([0,1]\times S^3, [0,1]\times K_0)$, completing the proof.
\end{proof}

\begin{rem}\label{rem:end}The above proof works for strongly homotopy-ribbon concordances in an $h$-cobordism $W$ from $S^3$ to itself. Upon disentangling $K_0$ from the attaching spheres for the complement of $\bar{C}\circ C$, we are left with the complement of the concordance $[0,1]\times K_0$ inside of $([0,1]\times S^3)\# \hat{W}$, where $\hat{W}$ is the homotopy 4-sphere obtained by filling in $\d W$ with two 4-balls. By factoring the cobordism map through the connected sum sphere, the induced map is easily seen to be the identity.
\end{rem}

\begin{figure}
\centering
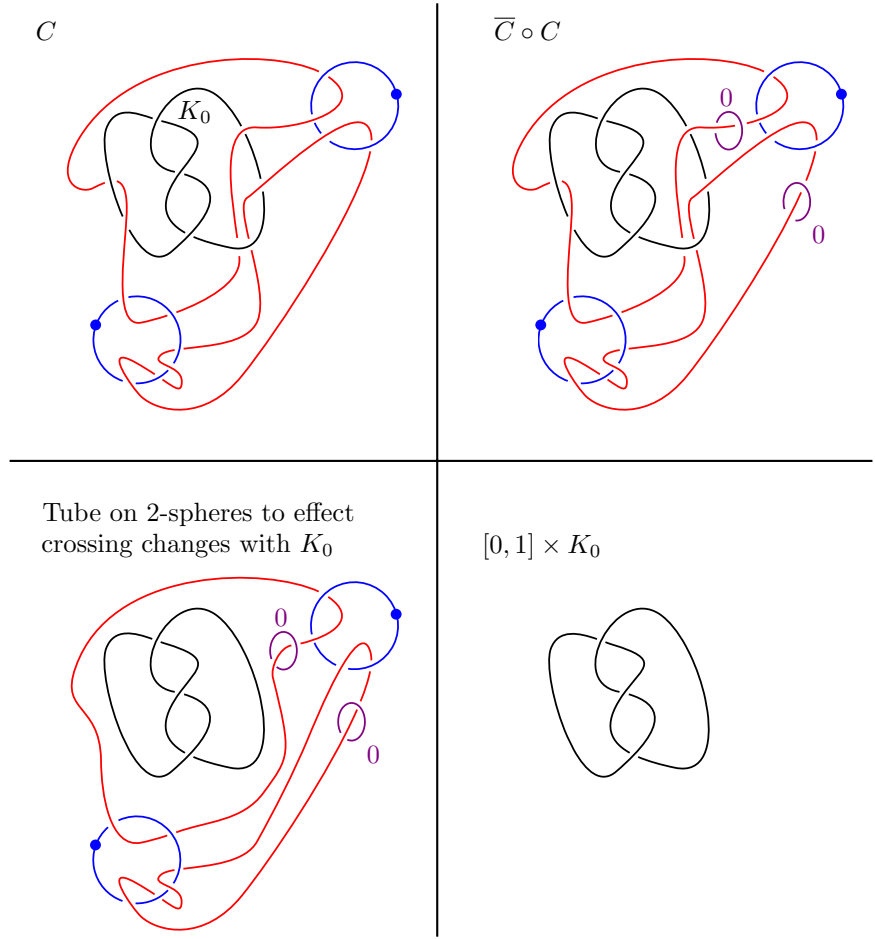
\caption{The proof of Theorem~\ref{thm:1}. Top left: the concordance $C$. Top right: the concordance $\bar{C}\circ C$, as in Lemma~\ref{lem:handle-decomp}. We do not picture the $3$-handle attaching spheres, which lie close to the (long, red) $2$-handle attaching circles. Bottom left: we tube $2$-spheres to $\bar{C}\circ C$ to effect crossing changes in this description. Bottom right: we obtain the trivial concordance $[0,1]\times K_0$.}\label{fig:example}
\end{figure}

\section*{Acknowledgments}

We would like to thank Andr\'{a}s Juh\'{a}sz for helpful conversations, in particular for pointing out that the tubing operation from \cite{ZemRibbon} can be interpreted in terms of sliding a surface over a 2-sphere. We would also like to thank the anonymous referee for some helpful comments.

\bibliographystyle{mrl}
\bibliography{biblio}

\end{document}